\documentclass[12pt]{amsart}
\title{Erd\H os and arithmetic progressions}
\author{W.T. Gowers}

\usepackage{amsmath, wrapfig}
\usepackage{dsfont, a4wide, amsthm, amssymb, amsfonts, graphicx}
\usepackage{fancyhdr, xspace, psfrag, setspace, supertabular, color}

\begin{document}

\newtheorem{theorem}{Theorem}[section]
\newtheorem{proposition}[theorem]{Proposition}
\newtheorem{lemma}[theorem]{Lemma}
\newtheorem*{claim}{Claim}
\newtheorem{corollary}[theorem]{Corollary}
\newtheorem{conjecture}[theorem]{Conjecture}
\newtheorem{definition}[theorem]{Definition}
\newtheorem{problem}[theorem]{Problem}
\newtheorem{example}[theorem]{Example}
\newtheorem{question}[theorem]{Question}
\newtheorem{remark}[theorem]{Remark}

\def\e{\epsilon}
\def\a{\alpha}
\def\b{\beta}
\def\d{\delta}
\def\cf{\mathcal{F}}
\def\E{\mathbb{E}}
\def\N{\mathbb{N}}
\def\R{\mathbb{R}}
\def\C{\mathbb{C}}
\def\Z{\mathbb{Z}}
\def\F{\mathbb{F}}
\def\hf{\hat f}

\onehalfspacing

\begin{abstract}
Two of Erd\H os's most famous conjectures concern arithmetic progressions. In this paper we discuss some of the progress that has been made on them.
\end{abstract}

\maketitle

\section{Introduction}

Possibly the best known of all of Erd\H os's many conjectures is the following striking statement.

\begin{conjecture} \label{sumofreciprocals}
Let $A$ be a set of positive integers such that $\sum_{n\in A}n^{-1}=\infty$. Then $A$ contains arbitrarily long arithmetic progressions.
\end{conjecture}

\noindent This conjecture is still wide open. Indeed, it is not even known whether $A$ must contain an arithmetic progression of length 3. 

There is another conjecture of Erd\H os about arithmetic progressions. It is not as famous as the first, but it is still well known and extremely interesting. It is sometimes referred to as \textit{Erd\H os's discrepancy problem}.

\begin{conjecture} \label{discrepancy}
Let $\e_1,\e_2,\e_3,\dots$ be a sequence taking values in the set $\{-1,1\}$. Then for every constant $C$ there exist positive integers $n$ and $d$ such that $|\sum_{m=1}^n\e_{md}|\geq C$.
\end{conjecture}

The purpose of this paper is to say a little bit about the two conjectures and to discuss some known results and related problems.

\section{Arithmetic progressions in sparse sets}

What does it tell us about a set $A$ if $\sum_{n\in A}n^{-1}$ diverges? Clearly it tells us that in some sense $A$ is not too small, since the larger it is, the more likely the sum of its reciprocals is to diverge. A rough interpretation of the condition turns out to be that the density $\d(n)=n^{-1}|A\cap\{1,2,\dots,n\}|$ decreases not too much faster than $(\log n)^{-1}$. One way of seeing this is as follows. Writing $1_A$ for the characteristic function of $A$, we have the trivial identity
\[1_A(n)=n\d(n)-(n-1)\d(n-1),\]
from which (if we adopt the convention that $\d(0)=0$) it follows that 
\[\sum_{n\in A}n^{-1}=\sum_nn^{-1}1_A(n)=\sum_n(\d(n)-\d(n-1)+\d(n-1)/n)=\sum_n\d(n-1)/n.\]
Thus, if the density decreases like $(\log n)^{-1}$ then we get a sum like $\sum_n1/n\log n$, which diverges, while if it decreases like, say, $(\log n)^{-1}(\log\log n)^{-2}$, then we get a convergent sum. 

Of course, the density does not have to decrease smoothly in this way, but this nevertheless gives a good general picture of what the conjecture is saying. In particular, the simple calculation just given tells us that if $\sum_{n\in A}n^{-1}=\infty$, then there must be infinitely many $n$ for which $\d(n)\geq (\log n)^{-1}(\log\log n)^{-2}$, so to prove Erd\H os's conjecture it is sufficient to prove the following statement.

\begin{conjecture}
For every $k$ there exists $n$ such that if $A$ is any subset of $\{1,\dots,n\}$ of cardinality at least $n/\log n(\log\log n)^2$, then $A$ contains an arithmetic progression of length~$k$.
\end{conjecture}

It is also not hard to show that to \emph{disprove} Erd\H os's conjecture, it would be sufficient to show that for every $k$ and every sufficiently large $n$ there exists a subset $A\subset\{1,\dots,n\}$ of cardinality at least $n/\log n$ that does not contain an arithmetic progression of length $k$. To do this, for each sufficiently large $r$ let $A_r$ be a subset of $\{2^r+1,\dots,2^{r+1}\}$ of size at least $cr^{-1}2^r$ that contains no arithmetic progression of length $k$ and let $A$ be the infinite set $A_s\cup A_{s+2}\cup A_{s+4}\cup\dots$ for a sufficiently large $s$. Then for every sufficiently large $n$ we have $\d(n)\geq c'(\log n)^{-1}$ and $A$ contains no arithmetic progression of length $k$. 
\medskip

Thus, Erd\H os's conjecture is basically addressing the following problem, and suggesting an approximate answer.

\begin{problem}
Let $k$ and $n$ be positive integers. How large does a subset $A\subset\{1,2,\dots,n\}$ have to be to guarantee that it contains an arithmetic progression of length $k$?
\end{problem}

\noindent The suggested answer is that a cardinality of somewhere around $n/\log n$ should be enough.

A natural starting point would be to prove \emph{any} bound of the form $o(n)$. This gives us another famous conjecture of Erd\H os, made with Paul Tur\'an in 1936 \cite{erdosturan}.

\begin{conjecture} \label{erdosturan}
For every positive integer $k$ and every $\d>0$ there exists $n$ such that every subset $A\subset\{1,2,\dots,n\}$ of cardinality at least $\d n$ contains an arithmetic progression of length $k$.
\end{conjecture}

Even this much weaker conjecture turned out to be very hard, and very interesting indeed: it can be seen as having given rise to several different branches of mathematics.

The first progress on the Erd\H os-Tur\'an conjecture was due to Roth, who proved in 1953 that it is true when $k=3$ \cite{roth}. Roth's proof, which used Fourier analysis, showed that $\d$ could be taken to be $C/\log\log n$ for an absolute constant $C$. The problem for longer progressions turned out to be much harder, and it was not until 1969 that there was further progress, when Szemer\'edi proved the result for $k=4$ \cite{szem1}, this time with a bound for $\d$ that was too weak to be worth stating explicitly. And a few years later (the paper was published in 1975), Szemer\'edi managed to prove the general case \cite{szem2}.

\subsection{Other proofs of Szemer\'edi's theorem}

This result was hailed at the time and is still regarded as one of the great mathematical results of the second half of the twentieth century, but it was by no means the end of the story: over the last four decades its significance has steadily grown. In this respect, the Erd\H os-Tur\'an conjecture is like many conjectures of Erd\H os. Initially it seems like an amusing puzzle, but the more you think about it, the more you come to understand that the ``amusing puzzle" is a brilliant distillation of a much more fundamental mathematical difficulty. There are few direct applications of Szemer\'edi's theorem (though they do exist), but an enormous number of applications of the methods that Szemer\'edi developed to prove the theorem, and in particular of his famous regularity lemma.

Since then, there have been several other proofs of the theorem, which have also introduced ideas with applications that go well beyond Szemer\'edi's theorem itself. In 1977, Furstenberg pioneered an ergodic-theoretic approach \cite{furst}, giving a new proof of the theorem and developing a method that went on to yield the first proofs of many generalizations, of which we mention three notable ones.

The first is a natural multidimensional version of Szemer\'edi's theorem, due to Furstenberg and Katznelson \cite{FK1}.

\begin{theorem}
For every $\d>0$, every positive integer $d$ and every subset $K\subset\Z^d$ there exists $n$ such that every subset $A\subset\{1,\dots,n\}^d$ of size at least $\d n^d$ contains a homothetic copy of $K$: that is, a set of the form $aK+b$ for some positive integer $a$ and some $b\in\Z^d$.
\end{theorem}

Next, we have the ``density Hales-Jewett theorem", also due to Furstenberg and Katznelson \cite{FK2}. For this we need a definition. If $x$ is a point in $\{1,\dots,k\}^n$ and $E$ is a subset of $\{1,2\dots,n\}$, then for each $1\leq j\leq k$ let $x\oplus jE$ be the point $y\in\{1,\dots,k\}^n$ such that $y_i=j$ for every $i\in A$ and $y_i=x_i$ otherwise. A \emph{combinatorial line} in $\{1,\dots,k\}^n$ is a set of points of the form $\{x\oplus jE:j=1,\dots,k\}$.

\begin{theorem}
For every $\d>0$ and every $k$ there exists $n$ such that every subset $A\subset\{1,\dots,k\}^n$ of cardinality at least $\d k^n$ contains a combinatorial line.
\end{theorem}

Finally, the Bergelson-Leibman theorem \cite{BL} is the following remarkable ``polynomial version" of Szemer\'edi's theorem.

\begin{theorem}
For every $\d>0$ and every sequence $P_1,\dots,P_k$ of polynomials with integer coefficients and no constant term there exists $n$ such that every subset $A\subset\{1,2,\dots,n\}$ of cardinality at least $\d n$ contains a subset of the form $\{a+P_1(d),a+P_2(d),\dots,a+P_k(d)\}$ with $d\ne 0$. 
\end{theorem}

\noindent If we take $P_i(d)$ to be $(i-1)d$, then we recover Szemer\'edi's theorem, but this result is considerably more general. For example, amongst many other things it implies that in Szemer\'edi's theorem we can ask for the common difference of the arithmetic progression we obtain to be a perfect cube.

Another approach to Szemer\'edi's theorem was discovered approximately twenty years later by the author \cite{G1,G2}. One of the reasons that Roth's proof for progressions of length 3 was not quickly followed by a proof of the general case was that while the number of arithmetic progressions of length 3 in a set can be expressed very nicely in terms of Fourier coefficients, there is no useful Fourier expression for the number of arithmetic progressions of length 4 (or more). The proofs in \cite{G1,G2} replaced the trigonometric functions that Roth used by polynomial phase functions (that is, functions of the form $\exp(2\pi i p(x))$ for some polynomial $p$) restricted to arithmetic progressions. This strongly suggested that there should be a kind of ``higher-order Fourier analysis", and, in a major recent achievement, such a theory was worked out by Green, Tao and Ziegler \cite{GTZ} (see also \cite{GT2,BTZ}. Their \emph{inverse theorem for the uniformity norms} had a very important application that we shall describe briefly later.

A fourth approach to the theorem had its roots in a fascinating argument of Ruzsa and Szemer\'edi \cite{RS}, who used Szemer\'edi's regularity lemma to prove the following result, which is now known as the \emph{triangle removal lemma}.

\begin{theorem}
For every $\e>0$ there exists $\d>0$ such that if $G$ is any graph with $n$ vertices and at most $\d n^3$ triangles, then there is a triangle-free graph that differs from $G$ by at most $\e n^2$ eges.
\end{theorem}

By applying the triangle removal lemma to a suitably chosen graph, one can deduce Roth's theorem (with a much worse bound). 

It is natural to wonder whether this idea can be generalized to give a proof of the general case of Szemer\'edi's theorem. This thought led R\"odl to formulate an approach to the theorem, in which the regularity lemma was generalized from graphs to hypergraphs. The generalization is not straightforward to state, and proving both it and an associated ``counting lemma" turned out to be hard. Frankl and R\"odl proved a hypergraph regularity lemma in 1992 \cite{FR1} and in 2002 managed to use it to prove Szemer\'edi's theorem for progressions of length 4 \cite{FR2}. The general case was proved by this method in independent work of Nagle, R\"odl and Schacht \cite{NRS} and the author \cite{G3}. (In the latter proof the formulation of the hypergraph regularity lemma was different, which made it harder to prove but made the counting lemma easier to prove.) Hypergraph regularity has gone on to have several other applications. 

An important development in our understanding of the regularity lemma came with work of Lov\'asz and others on \emph{graph limits}. Loosely speaking, with the help of the regularity lemma one can show that very large graphs look like measurable functions from $[0,1]^2$ to $[0,1]$. In a way this is not too surprising, because the regularity lemma allows one to approximate any graph with just a bounded amount of information about densities between subsets. What is more surprising, however, is that the graph-limits point of view leads to a simpler proof of the regularity lemma itself \cite{LS}: for the limiting arguments one can use a weaker regularity lemma, and once one has passed to a measurable function on $[0,1]^2$, one has a limit of step functions, which implies that if one partitions into a very fine grid, then the function will be approximately constant on most squares. 

Once one is given the statement of Szemer\'edi's regularity lemma and the basic idea of the standard proof, working out the details is not especially hard to begin with. However, the limits approach generalizes to hypergraphs \cite{ES}, where proving corresponding results is much harder, and gives rise to similar simplifications. The resulting hypergraph-limits approach to Szemer\'edi's theorem has a strong claim to be the simplest known proof of the theorem. More generally, graph and hypergraph limits have become a very active area of research with several other applications.

We briefly mention one other candidate for the simplest known proof of Szemer\'edi's theorem, which is a combinatorial proof of the density Hales-Jewett theorem, discovered by a ``massive online collaboration" \cite{polymath}. It is easy to see that the density Hales-Jewett theorem implies Szemer\'edi's theorem: one just needs to interpret the points in $\{1,\dots,k\}^n$ as base-$k$ representations of integers, and then every combinatorial line is an arithmetic progression of length $k$ (but not vice versa). Recently, this proof has been simplified yet further~\cite{DKT}.

\subsection{Quantitative considerations}

As we saw earlier, Conjecture \ref{sumofreciprocals} is roughly saying that a density of $(\log n)^{-1}$ is enough to guarantee an arithmetic progression. But what is special about this bound? Indeed, \emph{is} it special?

There are two sensible answers to this question: yes and no. The reason the bound is special, and the reason that Erd\H os asked the question, is that the primes have density around $(\log n)^{-1}$ in the first $n$ integers. One of Erd\H os's formative mathematical experiences was proving for himself that the sum of the reciprocals of the primes diverges, and it is clear that his main motivation for the sum-of-reciprocals conjecture was that it would imply that the prime numbers contain arbitrarily long arithmetic progressions. This would be an example of a result of a kind that Erd\H os particularly liked: a result that appears to be number-theoretic but turns out to be true for purely combinatorial reasons.

It would have been fascinating to know how Erd\H os would have reacted to the proof by Green and Tao \cite{GT1} that the primes do indeed contain arbitrarily long arithmetic progressions. In fact, Green and Tao proved the following stronger result.

\begin{theorem} 
For every $\d>0$ and every $k$ there exists $n$ such that if $A$ is any set of at least $\d n/\log n$ primes between 1 and $n$, then $A$ contains an arithmetic progression of length~$k$.
\end{theorem}

That is, not only do the primes contain arbitrarily long arithmetic progressions, but so does any subset of the primes of positive relative density. (Of course, this too is implied by the sum-of-reciprocals conjecture.)

The proof of this celebrated result did not go according to Erd\H os's plan, in that it made significant use of distribution properties of the primes. However, despite this, it would almost certainly have appealed to Erd\H os's love of combinatorial arguments, since the main new ingredient in the proof was in a sense ``purely combinatorial": they proved a ``relative version" of Szemer\'edi's theorem, showing that a set $A$ that is a relatively dense subset of a set $B$ must contain an arithmetic progression of length $k$, provided that $B$ is sufficiently large and sufficiently ``pseudorandom" in a technical sense that they defined. (The result they stated and used was actually more general than this: $B$ was replaced by a ``pseudorandom measure".) In order to prove this result, they used Szemer\'edi's theorem as well as techniques from several of the proofs of the theorem. Thus, the work on the Erd\H os-Tur\'an conjecture did in the end result in a solution to the problem that so fascinated Erd\H os.

Green and Tao followed this theorem with a project to obtain asymptotic bounds for the number of arithmetic progressions of length $k$ (and many other configurations) in the primes up to $n$. Over several years, they published a sequence of major papers, culminating in a proof, with Tamar Ziegler, of the inverse theorem for the uniformity norms \cite{GTZ}, mentioned earlier, at which point the project was completed.

\subsubsection{How natural is Erd\H os's conjecture?}

The fact that Erd\H os's conjecture implies an extremely striking result about the primes is not really evidence that the correct bound in Szemer\'edi's theorem is anywhere near $\d=(\log n)^{-1}$. Obtaining such a bound would be wonderful, but there is no strong reason to suppose that it would be the last word on the subject.

In particular, the best known \emph{lower} bound for Szemer\'edi's theorem is far smaller than $(\log n)^{-1}$. It comes from a construction of Behrend in 1946 \cite{behrend}. Behrend started from the observation that the surface of a sphere contains no three points in a line, and in particular no three points such that one is the midpoint of the other two. The argument proceeds as follows. For suitable integers $m$ and $d$, to be optimized at the end of the argument, one shows by the pigeonhole principle that there exists $r$ such that the sphere of radius $r$ contains many points in the grid $\{1,\dots,m\}^d$. Next, one embeds that grid ``isomorphically" into the set $\{1,2,\dots,(2m)^d\}$ by thinking of the points in $\{1,\dots,m\}^d$ as base-$2m$ representations of integers. The main property of this ``isomorphism" is that it does not create any arithmetic progressions of length 3 that were not present before. Finally, one maximizes the number of points in the spherical surface subject to the constraint that $(2m)^d=n$. The resulting bound is $\d=\exp(-c\sqrt{\log n})$.

This bound helps to explain why it is so hard to determine optimal bounds for Szemer\'edi's theorem, even when the progressions have length 3. On a first acquaintance with the problem, it is natural to conjecture that the extremal example would be given by a simple probabilistic construction. If that were the case, then there would be hope of proving that that construction was best possible by showing that ``quasirandom sets are best". An approach like this works, for example, if one wishes to minimize, for a given cardinality of a subset $A\subset\Z/n\Z$, the number of quadruples $(a_1,a_2,a_3,a_4)\in A^4$ such that $a_1+a_2=a_3+a_4$, at least when that cardinality is significantly greater than $\sqrt n$. However, random sets do not work for progressions of length 3: the standard method of choosing points randomly with probability $p$, where $p$ is chosen such that the expected number of progressions of length 3 is at most half the expected number of points, and then deleting a point from each progression, gives a lower bound of $\d=cn^{-2/3}$, far smaller than the Behrend bound. 

The Behrend bound can be slightly improved when the progressions are longer, but for now let us focus on progressions of length 3. What is the correct bound for the first non-trivial case of Szemer\'edi's theorem? This is a fascinating question that is still wide open, despite the attention of many mathematicians. However, there has been some very interesting progress.

As mentioned earlier, the original argument of Roth gave an upper bound of $C(\log\log n)^{-1}$. This bound was improved to one of the form $(\log n)^{-c}$ by Heath-Brown \cite{HB} and Szemer\'edi \cite{szem3}. An important new technique, the use of regular Bohr sets, was introduced by Bourgain in 1999 \cite{bour1}, to improve the constant $c$. More precisely, he obtained a bound of $C(\log\log n/\log n)^{1/2}$. A difficulty with the problem is that cyclic groups are not rich in subgroups, so dropping down to a subgroup is not an option. Regular Bohr sets are a kind of substitute for subgroups, allowing Bourgain to get round this difficulty. They have subsequently been used in  many other proofs.

For a while, Bourgain's result was seen as the limit of what could be achieved without a radical change of approach. It therefore came as a surprise in 2008 when Bourgain introduced an idea that allowed him to carry out the general scheme of his proof more efficiently and obtain a power of $2/3$ instead of $1/2$. Sanders \cite{sanders1} pushed this approach further and obtained a power of $3/4$. 

Sanders followed up this improvement with a major advance on the problem \cite{sanders2}. He found an argument that was substantially different from Bourgain's and used it to obtain a bound of $C(\log\log n)^5/\log n$. Thus, he was tantalizingly close to the logarithmic barrier. In fact, even a bound of $c\log\log n/\log n$ would be enough to prove purely combinatorially that the primes contain infinitely many arithmetic progressions of length 3, since if $m$ is a number with many small prime factors, then most arithmetic progressions with common difference $m$ contain almost no primes, which means that some have a high density of primes. Working out the details, one can find arithmetic progressions of length $n$ in which the primes have density $c\log\log n/\log n$. 

\subsubsection{What is the right bound for Roth's theorem?}

That is where things stand today. Is the Behrend bound correct, or is Sanders's upper bound close to optimal? Nobody knows, but there there are two recent results that give weakish evidence that the Behrend bound is more like the truth of the matter.

The first of these concerns a closely related problem about subsets of $\F_3^n$ (where $\F_3$ is the field with three elements). How large must a subset of $\F_3^n$ be to guarantee that it contains an affine line, or equivalently three points $x,y,z$ such that $x+y+z=0$? (Such a triple can also be thought of as an arithmetic progression, since if $x+y+z=0$, then $2y=x+z$.)

It was observed by Meshulam that Roth's original argument works very cleanly in this context (the main reason being that, in contrast with the cyclic group $\Z/n\Z$, the group $\F_3^n$ is very rich in subgroups), and yields the following theorem \cite{meshulam}.

\begin{theorem}
There exists a constant $C$ such that every subset $A\subset\F_3^n$ of density at least $C/n$ contains an affine line.
\end{theorem}

\noindent Thus, in this context, we have a logarithmic bound (since $n$ is logarithmic in the size, $3^n$, of the set $\F_3^n$). 

The gap between this and the best known lower bound is even more embarrassingly large than it is for Roth's theorem, since the lower bound is of the form $\alpha^n$ for some constant $\alpha<3$. (To obtain such a lower bound, one finds a low-dimensional example and takes powers of that example.) 

It was felt by many people that this was a better problem to attack than attempting to improve the bounds in Roth's theorem, since working in the group $\F_3^n$ presented technical simplifications without avoiding the deeper mathematical difficulties. And yet, despite the simplicity of the arguments for both the upper and lower bounds, for many years nobody could come up with any improvement. There was therefore considerable excitement in 2011 when Bateman and Katz \cite{BK} broke the logarithmic barrier for this problem, improving the upper bound to $C/n^{1+\e}$ for a small but fixed positive $\e$. Initially there was a hope that it might be possible to combine their ideas with those of Sanders to break the logarithmic barrier in Roth's theorem as well, thereby proving the first non-trivial case of Erd\H os's sum-of-reciprocals conjecture, but unfortunately good reasons emerged to suppose that this cannot be done without significant new ideas. However, the fact remains that the logarithmic barrier is not the right bound for the $\F_3^n$ version of the problem, which makes it hard to think of a good reason for its being the right bound for Roth's theorem itself.

The second recent result, also from 2011, makes it look as though a Behrend-type bound might be correct. Roth's theorem can be thought of as a search for solutions to the equation $x+z=2y$. Schoen and Shkredov, building on the methods that Sanders introduced to prove his near-logarithmic bound for Roth's theorem, showed that if we generalize this equation, then we can obtain a much better bound \cite{SS}.

\begin{theorem}
Let $A$ be a subset of $\{1,2,\dots,n\}$ of density $\exp(-c(\log n)^{1/6-\e})$. Then $A$ contains distinct elements $x_1,x_2,x_3,x_4,x_5$ and $y$ such that $x_1+x_2+x_3+x_4+x_5=5y$.
\end{theorem}

\noindent Note that the Behrend lower bound is easily adapted to this equation (since if $x_1,\dots,x_5,y$ are distinct and satisfy that equation then they cannot all lie on the surface of a sphere), so this result is within spitting distance of best possible.

Of course, one could state an Erd\H os-like corollary to this theorem: if $A$ is a set of integers such that $\sum_{n\in A}n^{-1}$ diverges, then $A$ contains a non-degenerate solution to the equation $x_1+x_2+x_3+x_4+x_5=5y$. However, the original result is more natural.

The result of Schoen and Shkredov is by no means conclusive evidence that the correct bound for Roth's theorem is of the form $\exp(-(\log n)^c)$, since convolutions of three or more functions are significantly smoother than convolutions of two functions, a phenomenon that also explains why the twin-prime conjecture and Goldbach's conjecture are much harder than Vinogradov's three-primes theorem. However, one can at least say, in the light of this result and the result of Bateman and Katz, that there is a significant chance that the logarithmic barrier for Roth's theorem will eventually be surpassed and the first non-trivial case of Erd\H os's conjecture proved. 

\subsubsection{Arithmetic progressions of length 4 or more}

What happens for longer progressions? As mentioned earlier, the bounds coming from Szemer\'edi's proof are very weak. Furstenberg's proof was infinitary and gave no bound at all (though a discrete version of his argument was later found by Tao \cite{T}, which in principle gave a weak quantitative bound). The first argument to give a ``reasonable" bound was the one in \cite{G1,G2}, where the following theorem was proved.

\begin{theorem}
Let $A$ be a subset of $\{1,2,\dots,n\}$ of density at least $C(\log\log n)^{-1/2^{2^{k+9}}}$. Then $A$ contains an arithmetic progression of length $k$.
\end{theorem}

\noindent Green and Tao subsequently improved the bound for $k=4$ to $\exp(-c\sqrt{\log\log n})$ \cite{GT2}. And that is the current state of the art, though for a finite-field analogue of the problem (again with $k=4$) they have a bound of the form $\exp(-(\log n)^c$ \cite{GT3}.

Will Erd\H os's sum-of-reciprocals conjecture be proved any time soon? There seems at least a fair chance that the case $k=3$ will be established within, say, the next ten years. There are significant extra difficulties involved when the progressions are longer, but a significant amount of technology for dealing with longer progressions has now been developed. Whether a bound for $k=3$ will lead to a bound for longer progressions probably depends a lot on what the proof for $k=3$ looks like, and by how much it beats the logarithmic bound. It may also depend on whether the inverse theorem for uniformity norms can be proved with good quantitative bounds.

\section{Erd\H os's discrepancy problem}

Let us now turn to Conjecture \ref{discrepancy}. Discrepancy problems are problems that ask how ``balanced" a colouring of a set can be with respect to some class of subsets. If we have a red/blue colouring $\kappa$ of a set $X$ and $A\subset X$, then define the discrepancy $\mathrm{disc}(\kappa,A)$ of $\kappa$ on $A$ to be the difference between the number of red elements of $A$ and the number of blue elements of $A$. The discrepancy $\mathrm{disc}(\kappa,\mathcal{A})$ of $\kappa$ with respect to $\mathcal{A}$ is then $\max_{A\in\mathcal{A}}\mathrm{disc}(\kappa,A)$. The discrepancy problem for $\mathcal{A}$ is the problem of determining the minimum of $\mathrm{disc}(\kappa,\mathcal{A})$ over all 2-colourings $\kappa$. We can of course think of $\kappa$ as a function from $X$ to $\{-1,1\}$ and then $\mathrm{disc}(\kappa,A)$ is $|\sum_{x\in A}\kappa(x)|$. The Erd\H os discrepancy problem is the discrepancy problem for the set $\mathcal{A}$ of \emph{homogeneous arithmetic progressions}: that is arithmetic progressions of the form $(d,2d,3d,\dots,md)$.

\subsection{Known bounds}

As with Szemer\'edi's theorem, it is tempting to conjecture, again wrongly, that random examples are best for this problem. If we choose a random sequence $(\e_i)$ of 1s and -1s, then the expected size of $\sum_{m=1}^n\e_{md}$ is around $\sqrt{n}$, and occasionally the size will be slightly bigger by a logarithmic factor. 

A simple example that gives rise to much slower growth of these sums is the following, observed by Borwein, Choi and Coons \cite{BCC}. Every positive integer $m$ can be written in a unique way as $(3a\pm 1)3^b$ for integers $a$ and $b$. We let $\e_m=1$ if $m$ is of the form $(3a+1)3^b$ and $-1$ if $m$ is of the form $(3a-1)3^b$. Note that this function is \emph{completely mulitplicative}: $\e_m\e_n=\e_{mn}$ for any two positive integers $m$ and $n$. Therefore, $|\sum_{m=1}^n\e_{md}|=|\e_d\sum_{m=1}^n\e_m|=|\sum_{m=1}^n\e_m|$ for any $n$ and $d$, so analysing the example reduces to calculating the rate of growth of the partial sums of the sequence.

To do this, we partition the integers from 1 to $n$ according to the highest power of $3$ that divides them. Let $A_{b,n}$ be the set of multiples of $3^b$ that are at most $n$ and are not multiples of $3^{b+1}$. Then $\sum_{m\in A_{b,n}}\e_m=1$ if in the ternary representation of $n$ the digit corresponding to multiples of $3^b$ is 1, and 0 otherwise. It follows that $\sum_{m=1}^n\e_m$ is equal to the number of ternary digits of $n$ that are equal to 1. In particular, it has magnitude at most $\log_3n$, which is far smaller than $\sqrt{n}$.

In the light of that example, it is natural to investigate the following weakening of Erd\H os's discrepancy conjecture, which Erd\H os also asked.

\begin{conjecture} \label{multiplicative}
Let $\e_1,\e_2,\e_3,\dots$ be a completely multiplicative sequence taking values in the set $\{-1,1\}$. Then the partial sums $\sum_{m=1}^n\e_m$ are unbounded.
\end{conjecture}

\noindent Remarkably, this conjecture is also very much open. Later we shall discuss evidence that it may be more or less as hard as the discrepancy problem itself.

What about the other direction? The sequence $(1,-1,-1,1,-1,1,1,-1,-1,1,1)$ has length 11 and has discrepancy 1 (where by ``discrepancy" we mean discrepancy with respect to the set of all homogeneous arithmetic progressions). This turns out to be the longest such sequence \cite{mathias}. Surprisingly, the longest sequence with discrepancy 2 is \emph{much} longer: there are a very large number of sequences of length 1124 with discrepancy 2, and it appears that this is the longest that such a sequence can be, though this has not yet been definitively proved. These experimental results, and almost all of the observations that follow, were discovered by the participants in Polymath5, an online collaboration that attacked the Erd\H os discrepancy problem in 2010 \cite{polymath5}. The fact that these sequences are so long gives one reason that the problem is so hard: it is difficult to imagine what a proof would be like that shows that the discrepancy of a $\pm 1$ sequence tends to infinity with the length of the sequence, while failing to prove the false result that the discrepancy of a sequence of length 1000 is at least 3. 

That is not the only reason for the problem's being hard. Another reason is that it is not easy to turn the problem into an analytic one -- a technique that is extremely helpful for many other problems. It would be very nice if the result were true not because the sequence consists of 1s and -1s but merely because it is large in some appropriate sense: for example, perhaps any sequence with values in $[-1,1]$ such that the average magnitude of the terms is non-zero could be expanded in terms of some cleverly chosen orthonormal basis, and perhaps this would prove that its discrepancy was unbounded. But a very simple example appears to kill off this hope straight away: the discrepancy of the periodic sequence $1,-1,0,1,-1,0,\dots$ is 1, and yet the average magnitude of its terms is $2/3$. Later we shall see that this example is not quite as problematic as it at first appears. Note that this example is a Dirichlet character: it is intriguing that the ``difficult" examples we know of all seem to be built out of characters in simple ways.

\subsection{Variants of the conjecture}

Sometimes, a good way of solving a problem is to replace the statement you are trying to prove by something stronger. There are several promising strengthenings of the Erd\H os discrepancy conjecture. An obvious one is to replace $\pm 1$-valued sequences by sequences that take values in some more general set. The example presented shows that we have to be a little careful about this, but the following conjecture is a reasonable one, and is also open.

\begin{conjecture} \label{hilbertversion}
Let $x_1,x_2,\dots$ be a sequence of unit vectors in a (real or complex) Hilbert space. Then for every $C$ there exist $n,d$ such that $\|\sum_{m=1}^nx_{md}\|\geq C$.
\end{conjecture}

\noindent Since $\R$ is a Hilbert space, this conjecture is a generalization of Erd\H os's conjecture. A conjecture intermediate between the two is one where the $x_i$ are complex numbers of modulus 1.

A less obvious strengthening was formulated by Gil Kalai (one of the Polymath5 participants), and called the ``modular version" of the Erd\H os discrepancy problem.

\begin{conjecture} \label{modularversion}
For every prime $p$ there exists $N$ such that if $x_1,x_2,\dots,x_N$ is any sequence of non-zero elements of $\Z/p\Z$, then for every $r\in\Z/p\Z$ there exist $n$ and $d$ with $nd\leq N$ and $\sum_{m=1}^nx_{md}\equiv r$ mod $p$.
\end{conjecture}

If we insist that each $x_i$ is $\pm 1$ mod $p$, then the conjecture becomes obviously equivalent to the original Erd\H os problem. However, since the problem does not involve products of the $x_i$, there is nothing special about the numbers $\pm 1$, so in this context it becomes natural to replace the set $\{-1,1\}$ by the set of all non-zero elements. The motivation for this conjecture was the hope that the polynomial method might be applicable to it. So far this has not succeeded, but the modular version gives us a valuable new angle on the problem. 

A possible generalization of the modular version to composite moduli $m$ would be to ask that the $x_i$ are coprime to $m$ (which is obviously a necessary condition if we want to be able to produce all numbers $r$). For amusement only, we state another conjecture here. It is similar in spirit to the more general modular version, but not quite the same. 

\begin{conjecture}\label{irrationalversion}
Let $K$ be a finite set of irrational numbers and let $x_1,x_2,\dots$ be a sequence of elements of $K$. Then the sums $s_{n,d}=\sum_{m=1}^nx_{md}$ are dense mod 1.
\end{conjecture}

\noindent Note that the special case where $K$ is of the form $\{\a,-\a\}$ for an irrational number $\a$ is equivalent to the original discrepancy conjecture. It is not clear whether there are any logical relationships between Conjectures \ref{modularversion} and \ref{irrationalversion}.

\subsection{Some approaches to the conjecture}

Although the Erd\H os discrepancy problem looks very hard, there are some approaches that at least enable one to start thinking seriously about it. Here we discuss three of these approaches.

\subsubsection{Completely multiplicative sequences}

A close look at the very long sequences of discrepancy 2 that were produced experimentally reveals interesting multiplicative structure. The sequences are not completely multiplicative, but they appear to ``want" to have multiplicative features. For example, if you look at the values of a completely multiplicative $\pm 1$ sequence along a geometric progression, then they will either be constant or alternating. In the long sequences of discrepancy 2 we do not see that behaviour, but we do see quasiperiodic behaviour, at least for a while: towards the end, the patterns break down. There is a natural, but speculative, interpretation of this. The sequences appear to be some kind of ``projection" to the set of $\pm1$ sequences of highly structured sequences taking values in $\C$. Towards the end, if the structure is followed too closely, the discrepancy rises to 3, but for a while that can be countered by simply switching the signs of a few terms in the sequence. If those terms correspond to integers with not many factors, then not many homogeneous progressions are affected, so one can extend the length of the sequence by sacrificing the structure. But since it was the structure that allowed the sequence to get long in the first place, this process is eventually doomed: one has to make more and more ad hoc tweaks, and eventually it becomes impossible to continue.

This picture suggests the following line of attack. Perhaps one could attempt to show that the worst examples -- that is, the ones with lowest discrepancy -- have to have some kind of multiplicative structure. Then one could attempt to prove the easier (one hopes) statement that a sequence with multiplicative structure must have unbounded discrepancy.

An approach like this might seem a bit fanciful. Remarkably, however, there is a precise reduction from the Erd\H os discrepancy problem to a related problem about multiplicative sequences, discovered by Terence Tao (another Polymath5 participant). With the help of a few lines of Fourier analysis, he proved the following result \cite{polymath52}.

\begin{proposition}
Suppose that there exists an infinite $\pm 1$ sequence of discrepancy at most $C$. Then there exists a completely multiplicative sequence $z_1,z_2,\dots$ of complex numbers of modulus 1 such that the averages $N^{-1}\sum_{n=1}^N|\sum_{i=1}^nz_i|^2$ are bounded above by a constant depending on $C$.
\end{proposition}

Thus, to prove the Erd\H os discrepancy problem, it is enough to prove the following conjecture about completely multiplicative complex-valued sequences.

\begin{conjecture}
There exists a function $\omega:\N\to\R$ tending to infinity with the following property. Let $z_1,z_2,\dots$ be any completely multiplicative sequence $z_1,z_2,\dots$ of complex numbers of modulus 1. For each $n$ let $s_n$ be the $n$th partial sum of this sequence. Then $(|s_1|^2+\dots+|s_N|^2)/N\geq\omega(N)$ for every $N$.
\end{conjecture}

This is not quite the same as saying that every completely multiplicative sequence has unbounded discrepancy, even if we generalize to the complex case. What it says is not just that the \emph{worst} partial sums of such a sequence should be large, but that the \emph{average} partial sums should be large (uniformly over all such sequences). However, if the weaker statement is true, then it looks likely that the stronger statement will be true as well.

A pessimistic view of this reduction would be to say that it shows that the multiplicative problem is probably just as hard as the original. However, completely multiplicative sequences have so much more structure than arbitrary sequences that it is not clear that such pessimism is justified.

\subsubsection{Semidefinite programming}

The following very nice observation was made by Moses Charikar (yet another Polymath5 participant), which offers a way round the obstacle that the sequence $1,-1,0,1,-1,0,\dots$ has bounded discrepancy. 

\begin{proposition} \label{charikar}
Suppose that we can find non-negative coefficients $c_{m,d}$ for each pair of natural numbers $m$ and $d$, and a sequence $(b_n)$ such that $\sum_{m,d}c_{m,d}=1$, $\sum_nb_n=\infty$, and the real quadratic form 
\[\sum_{m,d}c_{m,d}(x_d+x_{2d}+\dots+x_{md})^2-\sum_nb_nx_n^2\]
is positive semidefinite. Then every $\pm 1$ sequence has unbounded discrepancy.
\end{proposition}

\begin{proof}
If $(\e_n)$ is a $\pm 1$ sequence, then the positive semidefiniteness of the quadratic form tells us that
\[\sum_{m,d}c_{m,d}(\e_d+\e_{2d}+\dots+\e_{md})^2\geq \sum_nb_n\e_n^2=\sum_nb_n\]
Since $\sum_{m,d}c_{m,d}=1$ and $\sum_nb_n=\infty$, it follows that the sums $\e_d+\dots+\e_{md}$ are unbounded.
\end{proof}

The same argument shows that if $\sum_nb_n=C$ then there exist $m,d$ such that $|\e_d+\e_{2d}+\dots+\e_{md}|\geq C^{1/2}$. It also proves the Hilbert-space version of the Erd\H os discrepancy conjecture, since if the $x_i$ are vectors in a Hilbert space, then the non-negative definiteness of the quadratic form implies that
\[\sum_{m,d}c_{m,d}\|x_d+x_{2d}+\dots+x_{md}\|^2-\sum_nb_n\|x_n\|^2\]
is non-negative (as can be seen by expanding out the norms and looking at each coordinate).

Less obviously, the existence of a quadratic form satisfying the conditions of Proposition \ref{charikar} is actually \emph{equivalent} to a positive solution to the Hilbert-space version of the conjecture. 

\begin{proposition} \label{charikarconverse}
Suppose that every infinite sequence of unit vectors in a real Hilbert space has unbounded discrepancy. Then for every $C$ there exists $N$, a set of non-negative coefficients $c_{m,d}$ for each pair of natural numbers $m$ and $d$ with $md\leq N$, and a sequence $(b_1,\dots,b_N)$ such that $\sum_{m,d}c_{m,d}=1$, $\sum_{n=1}^Nb_n\geq C$, and the real quadratic form 
\[\sum_{m,d}c_{m,d}(x_d+x_{2d}+\dots+x_{md})^2-\sum_nb_nx_n^2\]
is positive semidefinite.
\end{proposition}

\begin{proof}
For each $m,d$ with $md\leq N$ define $A_{m,d}$ to be the $N\times N$ matrix with $ij$th entry equal to 1 if both $i$ and $j$ belong to the arithmetic progression $\{d,2d,\dots,md\}$ and 0 otherwise. Then the conclusion tells us that there exists an $N\times N$ diagonal matrix with entries adding up to at least $C$ that can be written as a convex combination of the matrices $A_{m,d}$ minus a positive semidefinite matrix. If this cannot be done, then by the Hahn-Banach separation theorem there must be a functional that separates the convex set of diagonal matrices with entries adding up to at least $C$ from the convex set consisting of convex combinations of the $A_{m,d}$ minus positive semidefinite matrices. Let us regard this functional as an $N\times N$ matrix $B$ in the inner product space that consists of all $N\times N$ matrices with square-summable entries and the obvious inner product.

What properties must this matrix $B$ have? We may suppose that $\langle D,B\rangle\geq 1$ for every diagonal matrix with entries adding up to at least $C$ and $\langle A,B\rangle\leq 1$ whenever $A$ is a convex combination of the matrices $A_{m,d}$ minus a positive semidefinite matrix. The first condition implies that $B$ is constant on the diagonal and that the constant is at least $C^{-1}$.

The second condition implies that $B$ has non-negative inner product with every positive semidefinite matrix, since if $A$ were a counterexample, then we could make $\langle -\lambda A,B\rangle$ arbitrarily large and positive by taking $\lambda$ sufficiently large and positive. In particular, if $x\in\R^N$ and we take $A$ to be the positive semidefinite matrix $x\otimes x$ (that is, the matrix with $ij$th element $x_ix_j$), then $\langle x,Bx\rangle=\langle x\otimes x,B\rangle\geq 0$, so $B$ is itself positive semidefinite. This is well known to be equivalent to the assertion that there are vectors $v_1,\dots,v_N$ in an inner product space such that $B_{ij}=\langle v_i,v_j\rangle$ for every $i,j$. Since $B_{ii}=c\geq C^{-1}$ for every $i$, we find that each vector $v_i$ has norm $\sqrt c$. 

Finally, since the zero matrix is positive semidefinite, the second condition also implies that $B$ must have inner product at most 1 with each $A_{m,d}$. In terms of the vectors $v_i$, this is precisely the statement that $\|v_d+v_{2d}+\dots+v_{md}\|^2\leq 1$, as can be seen by expanding the left-hand side.

If we now rescale so that the $v_i$ become unit vectors, this last inequality changes to $\|v_d+v_{2d}+\dots+v_{md}\|^2\leq K$, for some constant $K\leq C$. 

Therefore, if the conclusion fails for some constant $C$, we can find, for each $N$ a sequence of $N$ unit vectors of discrepancy at most $\sqrt{C}$. After applying a suitable rotation, we may assume that for each $n$ the $n$th vector in this sequence is spanned by the first $n$ standard basis vectors of $\R^N$. Therefore, an easy compactness argument gives us an infinite sequence of unit vectors with discrepancy at most $\sqrt{C}$, a contradiction. 
\end{proof}

Recall that the problem with the sequence $1,-1,0,1,-1,0,\dots$ is that it is ``large" in a natural sense (namely having average magnitude bounded away from zero), but has bounded discrepancy. What Proposition \ref{charikar} tells us is that there is a chance of proving that every sequence that is large with respect to a suitable \emph{weighted} norm -- the weighted $\ell_2$-norm with weights $b_n$ -- has unbounded discrepancy. Thus, there is after all a way of making the problem analytic rather than purely combinatorial. 

What can we say about a set of weights that would work? The lesson of the troublesome $1,-1,0,1,-1,0,\dots$ example is that the weights should be concentrated on numbers with many factors. For example, if the sum of the $b_n$ over all non-multiples of 3 is infinite, then the weights cannot work, since then if $(x_n)$ is the troublesome sequence, we have $\sum_nb_nx_n^2=\infty$ and yet the discrepancy is finite. (This does not contradict Proposition \ref{charikar}: it just means that for this choice of $(b_n)$ we cannot find appropriate coefficients $c_{m,d}$.) 

It is not easy to write down a set of weights that has any chance of working -- in fact, that is worth stating as an open problem -- albeit not a wholly precise one.

\begin{problem}
Find a system of weights $(b_n)$ with $\sum_nb_n=\infty$ for which it is reasonable to conjecture that every sequence $(x_n)$ such that $\sum_nb_nx_n^2=\infty$ has unbounded discrepancy.
\end{problem}

One of the things that makes Proposition \ref{charikar} interesting is that it suggests a experimental line of attack on the Erd\H os discrepancy problem. First, one uses semidefinite programming to determine, for some large $N$, the sequence $(b_1,b_2,\dots,b_N)$ with largest sum such that the diagonal matrix with those weights can be written as a convex combination of the matrices $A_{m,d}$ minus a positive semidefinite matrix. Next, one stares hard at the sequence and tries to spot enough patterns in it to make a guess at an infinite sequence that would work. Finally, one attempts to decompose the corresponding infinite diagonal matrix (perhaps using the experimental values of the coefficients $c_{m,d}$ as a guide). 

Some efforts were made by Polymath5 participants in this direction, but so far they have not succeeded. One problem is that cutting off sharply at $N$ appears to introduce misleading ``edge-effects". But even if one finds ways of smoothing the cutoff, the experimental data is hard to interpret, though it certainly confirms the principle that the weights $b_n$ should be concentrated on positive integers $n$ with many factors. Another serious difficulty is that because we already know that there are very long sequences with small discrepancy, the matrices we find experimentally will have to be extremely large if they are to give us non-trivial lower bounds for discrepancy -- large enough that the semidefinite programming algorithms take a long time to run. Despite these difficulties, this still seems like a promising approach that should be explored further.

\subsubsection{Representing diagonal matrices}

We end by mentioning an approach based on an observation that is somewhat similar to Proposition \ref{charikar} but that does not involve the slightly tricky concept of positive semidefiniteness. This approach was again one of the fruits of the Polymath5 discussion.

Let us define a \emph{HAP matrix} to be a matrix $A$ of the following form. Take two homogeneous arithmetic progressions $P$ and $Q$ and define $A_{ij}$ to be $1$ if $i\in P$ and $j\in Q$ and 0 otherwise. In other words, a HAP matrix is the characteristic function of a product of two homogeneous arithmetic progressions. 

\begin{proposition} \label{diagonalrep}
Suppose that there exists an $N\times N$ diagonal matrix of trace at least $C$ that belongs to the symmetric convex hull of all HAP matrices. Then every $\pm 1$ sequence of length $N$ has discrepancy at least $\sqrt C$.
\end{proposition}

\begin{proof} Let the diagonal matrix $D$ have diagonal entries $b_1,\dots,b_N$ and suppose that it can be written as $\sum_i\lambda_iA_i$ with $\sum_i|\lambda_i|\leq 1$ and with each $A_i$ a HAP matrix. Let $\e=(\e_1,\dots,\e_N)$ be a $\pm 1$ sequence. Then
\[C\leq\sum_nb_n\e_n^2=\langle\e,D\e\rangle=\sum_i\lambda_i\langle\e,A_i\e\rangle\ .\]
It follows that there exists $i$ such that $|\langle\e,A_i\e\rangle|\geq C$. If $P$ and $Q$ are the HAPs from which $A_i$ is built, then 
\[\langle\e,A_i\e\rangle=(\sum_{i\in P}\e_i)(\sum_{j\in Q}\e_j)\ ,\]
which implies that at least one of $\sum_{i\in P}\e_i$ and $\sum_{j\in Q}\e_j$ has modulus at least $\sqrt{C}$.
\end{proof}

Once again, the argument generalizes easily to unit vectors in a Hilbert space. And again there is an implication in the other direction.

\begin{proposition} \label{representationconverse}
Let $C$ be a constant, let $N$ be a positive integer, and suppose that for every $N\times N$ real matrix $A=(a_{ij})$ with 1s on the diagonal there exist homogeneous arithmetic progressions $P$ and $Q$ such that $|\sum_{i\in P}\sum_{j\in Q}a_{ij}|\geq C$. Then there is a diagonal matrix of trace at least $C$ that belongs to the symmetric convex hull of all HAP matrices.
\end{proposition}

\begin{proof}
Again we use the Hahn-Banach theorem. If no such diagonal matrix exists, then there is a linear functional, which we can represent as taking the inner product with a matrix $A$, that separates diagonal matrices of trace at least $C$ from convex combinations of HAP matrices and minus HAP matrices. If $\langle D,A\rangle\geq 1$ for every diagonal matrix $D$ of trace at least $C$, then $A$ must be constant on the diagonal and the constant must be at least $C^{-1}$. And if $|\langle B,A\rangle|< 1$ for every HAP matrix $B$, then for any two homogeneous arithmetic progressions $P$ and $Q$ we have $|\sum_{i\in P}\sum_{j\in Q}a_{ij}|<1$. And now if we choose $\lambda$ such that $\lambda A$ has 1s along the diagonal, then the matrix $\lambda A$ contradicts our hypothesis.
\end{proof}

In the light of this proposition (which is easily seen to be an equivalence) it is natural to make the following conjecture, which is yet another strengthening of the Erd\H os discrepancy problem.

\begin{conjecture} \label{matrixdiscrepancy}
For every $C$ there exists $N$ such that if $A=(a_{ij})$ is any real $N\times N$ matrix with 1s on the diagonal, then there exist homogeneous arithmetic progressions $P$ and $Q$ such that $|\sum_{i\in P}\sum_{j\in Q}a_{ij}|\geq C$.
\end{conjecture}

If we apply that conjecture in the case where $a_{ij}=\e_i\e_j$ for some $\pm 1$ sequence $(\e_1,\dots,\e_N)$, then the conclusion is that $|\sum_{i\in P}\e_i\sum_{j\in Q}\e_j|\geq C$, from which it follows that the sequence has discrepancy at least $\sqrt{C}$. Thus, the conjecture really is a strengthening of the Erd\H os discrepancy conjecture. Indeed, given how much weaker the condition of having 1s on the diagonal is than the condition of being a tensor product of two $\pm 1$ sequences, it is a very considerable strengthening. And yet it still appears to have a good chance of being true.

\section{conclusion}

The aim of this paper has been to give some idea of what is currently known about two notable conjectures of Erd\H os concerning arithmetic progressions. It has therefore been more about questions than answers, but Erd\H os would have been the last person to mind that. I imagine him sitting with ``the book" open at the relevant page, smiling at us as we struggle to find the proofs that he is now able to enjoy.

\end{document}